\documentclass[letterpaper, 11pt, reqno]{amsart}
\usepackage[english]{babel}
\usepackage{graphicx,xcolor}
\usepackage[all]{xy}
\usepackage{amsfonts,dsfont,mathrsfs}
\usepackage{amsmath, amsthm, amssymb}
\usepackage{mathrsfs}
\usepackage{enumerate, url}
\usepackage[foot]{amsaddr}
\usepackage{fancyhdr}
\usepackage{hyperref}
\usepackage{ifthen,srcltx}
\usepackage{comment}
\usepackage{circuitikz}

\newtheorem*{theorem*}{Theorem}

\DeclareMathOperator{\diamop}{diam}

\usepackage{mathtools}
\DeclarePairedDelimiterX{\norm}[1]{\lVert}{\rVert}{#1}

\newcommand\FF{{\mathcal F}}

\newcommand\LL{{\mathcal L}}
\newcommand\MM{{\mathcal M}}

\newcommand\PP{{\mathcal P}}

\newcommand\PMF{{\PP\kern-2pt\MM\FF}}

\newcommand\PML{{\PP\kern-2pt\MM\LL}}

\newcommand\R{{\mathbb R}}

\renewcommand\P{{\mathbb P}}

\newcommand{\fsubd}{\mathrel{{\scriptstyle\searrow}\kern-1ex^d\kern0.5ex}}
\newcommand{\bsubd}{\mathrel{{\scriptstyle\swarrow}\kern-1.6ex^d\kern0.8ex}}
\newcommand{\fsubeq}{\mathrel{\raise-.7ex\hbox{$\overset{\searrow}{=}$}}}
\newcommand{\bsubeq}{\mathrel{\raise-.7ex\hbox{$\overset{\swarrow}{=}$}}}

\newcommand{\tsh}[1]{\left\{\kern-.9ex\left\{#1\right\}\kern-.9ex\right\}}

\newtheorem{thm}{Theorem}[section]
\newtheorem{prop}[thm]{Proposition}
\newtheorem{lemma}[thm]{Lemma}
\newtheorem{cor}[thm]{Corollary}

\newtheorem*{corollary*}{Corollary}

\theoremstyle{definition}
\newtheorem{dfn}[thm]{Definition}

\newtheorem{rem}[thm]{Remark}
\newtheorem{rems}[thm]{Remarks}

\newtheorem{claim}[thm]{Claim}

\newtheorem{question}[thm]{Question}

\newtheorem{case}{Case}

\DeclareMathOperator{\interior}{int}
\usepackage{mathabx}

\newsavebox{\commentbox}
\newenvironment{mycomment}%
{\ifthenelse{\equal{\showcomments}{yes}}%
{\footnotemark
        \begin{lrbox}{\commentbox}
        \begin{minipage}[t]{1.25in}\raggedright\sffamily\tiny
        \footnotemark[\arabic{footnote}]}
{\begin{lrbox}{\commentbox}}}
{\ifthenelse{\equal{\showcomments}{yes}}
{\end{minipage}\end{lrbox}\marginpar{\usebox{\commentbox}}}
{\end{lrbox}}}

\newcommand{\showcomments}{yes}

\begin{document}

\title[Title goes here]{$\beta$-Uniform Convexity and Divisible Domains}
\author[A. Pompilio]{Amelia Pompilio}

\date{\today}

\begin{abstract}

Divisible convex sets have long been important in the study of Hilbert geometries. When a divisible convex set is an ellipsoid, the Hilbert geometry it induces is the hyperbolic space. In general, strictly convex divisible domains exhibit negative curvature properties, but only the ellipsoid is a CAT(0) space. The notion of $p$-uniform convexity from the theory of Banach spaces has been proposed by Shin-Ichi Ohta as a generalization of the Alexandrov-Toponogov comparison theorems to Finsler manifolds. We prove that a natural Finsler metric on a strictly convex divisible domain is $\beta$-uniformly convex, where the constant $\beta$ is related to the regularity of the boundary. We use this to show, with AI assistance, that the Hilbert metric, under suitable local and scale-dependent assumptions, is $\beta$-uniformly convex on such domains.

\end{abstract}

\maketitle

\setcounter{tocdepth}{1}
\numberwithin{equation}{section}
\tableofcontents

\section{Introduction}
\label{sec:intro}

Divisible convex domains $\Omega \subset \R\P^n$ have historically been objects of interest due to their rich connections to multiple areas of geometry. They form the most interesting examples of Hilbert geometries and have been studied as in \cite{Benoist} and \cite{Marquis}. 

Kelly and Straus proved that if a Hilbert geometry has determinate Busemann curvature, then it must be an ellipsoid and hence hyperbolic \cite{KS1}. Further, as Hilbert geometries are Finsler manifolds, they are CAT($\kappa$) only when $\Omega$ is an ellipsoid \cite{Ohta}. For non-ellipsoid strictly convex divisible domains, there are still hints of negative curvature; for example, the geodesic flow of the Hilbert metric on the quotient $\Gamma \backslash \Omega$ is Anosov (see Theorem \ref{thm:anosov}). The goal of this paper is to explore the following question.

\begin{question}
    Is there a notion of curvature that strictly convex divisible domains satisfy?
\end{question}

To this end, we study the notion of $p $-uniform convexity, a decay condition on the modulus of convexity (see Definition \ref{def:puc}) coming from the theory of Banach spaces. In \cite{Ohta}, Ohta proposes this notion as a generalization of the CAT(0) condition.

For $x_0 \in \Omega$, let $M_\Omega(x_0, \cdot)$ be the Minkowski functional (see Definition \ref{dfn:minkowski}).

\begin{theorem*}[\ref{thm:mainminkowski}]
Let $\Omega \subset \mathbb{RP}^n$ be a strictly convex divisible domain. Then for any $x_0 \in \Omega,$ the space $\big(T_{x_0}\Omega, M(x_0, \cdot)\big)$ is $\beta$-uniformly convex for some $\beta \in [2, \infty)$.
\end{theorem*}

The proof of Theorem \ref{thm:mainminkowski} is divided into three cases; one degenerate case, and two cases in which we use geometric constructions to bound the modulus of convexity of $\Omega$ from below. All three cases require a series of affine transformations that situate $\Omega$ in convenient coordinates for our computations, and these transformations are outlined in Section \ref{sec:setup}.

There is a natural Finsler metric $F_\Omega(x,\cdot)$ on $\Omega$ that is the symmetrization of the Minkowski functional $M_\Omega(x_0, \cdot)$, see Definition \ref{def:finsler}. Due to this fact we get the following corollary of Theorem \ref{thm:mainminkowski}:

\begin{corollary*}[\ref{cor:mainfinsler}]
    Let $\Omega \subset \mathbb{RP}^n$ be a strictly convex divisible domain. Then the space $\big(T_x\Omega, F_\Omega(x, \cdot)\big)$ is $\beta$-uniformly convex for some $\beta \in [2, \infty)$.
\end{corollary*}

In Theorem \ref{thm:mainminkowski} and Corollary \ref{cor:mainfinsler}, we consider $\Omega$ as the unit ball of the tangent space $T_x\Omega$ and use the known results about the boundary regularity of $\Omega$ to study the behavior of the Finsler metric. With these results in hand, it is of particular interest to investigate the possible $\beta$-uniform convexity of the Hilbert metric itself. Inspired by Ohta's generalization of 2-uniform convexity for nonlinear metric spaces (see Definition \ref{nonlinearizedpuc}), we have the following result:

\begin{theorem*}[\ref{localishBUC}]
    Let $\Omega$ be a strictly convex divisible domain. Let $K \subset U \subset \Omega$, where $K$ is compact and $U$ is contained in a coordinate chart. Fix $R \geq 1.$ Then, for $C_0 \in (0,C),$ there exists $r>0$ such that the following holds.

Let $x, y, z\in \Omega$, let $\gamma:[0,1]\longrightarrow  \Omega$ be the geodesic from $y$ to $z$ parametrized by arclength, and let $m=\gamma\big(\frac{1}{2}\big)$. Set $\rho \coloneqq d_\Omega(y,z) < r,$ and assume that $m \in K$ and $d_\Omega(x,m) \leq R\rho$. Then

\begin{equation}
        \frac{C_0}{4}d_\Omega(y,z)^\beta \leq \frac{1}{2}d_\Omega(x,y)^\beta +
\frac{1}{2}d_\Omega(x,z)^\beta  - d_\Omega(x,m)^\beta
\end{equation}
\end{theorem*}

The main technical challenge in attempting to adapt Ohta's work to the case of strictly convex divisible domains is the comparative lack of regularity. Indeed, Ohta finds suitable conditions under which a smooth Finsler manifold is 2-uniformly convex, but the natural Finsler metric on a strictly convex divisible domain is only $C^\alpha$ for some $\alpha \in (1,2)$. As such, Ohta's methods, which use the Finsler version of the second variation formula and the existence of a fundamental tensor, are unavailable for the subset of divisible domains treated in this paper. Instead, we use pointed Gromov-Hausdorff convergence as the main tool for proving that Corollary \ref{cor:mainfinsler} implies some restricted $\beta$-uniform convexity of the Hilbert metric itself; this approach and subsequent proof was found with AI assistance.

This paper is organized as follows. In Section \ref{sec:background} we briefly introduce the notions of divisible convex sets, $\alpha$-H{\"o}lder regularity and $\beta$-convexity, and $p$-uniform convexity. In Section \ref{sec:setup} we describe the precise coordinate set-up required for the proof of Theorem \ref{thm:mainminkowski}. In Section \ref{sec:proof} we prove Theorem \ref{thm:mainminkowski} and Corollary \ref{cor:mainfinsler}. In Section \ref{sec:uniformuniform} we prove that Corollary \ref{cor:mainfinsler} can be extended to being true for any basepoint contained in a compact subset of $\Omega$. In Section \ref{sec:future} we define pointed Gromov-Hausdorff convergence and prove Theorem \ref{localishBUC}.

\subsection*{Use of AI} Section 6 was made possible through substantial interaction with ChatGPT 5.5. The model assisted with structuring the proofs, searching the literature, identifying errors in previous drafts, and pointing out inconsistencies in my notation. 

Sections 1 through 5 of this paper were first created without the use of generative AI tools. Before posting the updated version of this paper with Section 6 included, I ran the entire draft through ChatGPT 5.6 Sol to check for errors, and the model found a gap in the proof of Corollary \ref{cor:mainfinsler}. I rewrote the proof with the help of the model.

This paper does not contain AI generated text; all exposition is mine. All proofs are verified to the best of my ability.

\subsection*{Acknowledgments} I would like to thank Wouter Van Limbeek for his continued guidance, patience, and support. His expertise and engagement made this project, as well as my time at UIC, a delightful experience. I would also like to thank Beibei Liu and Jean Lafont for their ideas, suggestions, and time, and for warmly welcoming me into the Ohio State community.

\section{Background}\label{sec:background}

\subsection{Divisible Convex Sets}

\begin{dfn}[see e.g.\cite{Benoist}]
    A \emph{divisible convex set} is a properly convex open subset $\Omega \subset \mathbb{RP}^n$ for which there exists a discrete group $\Gamma$ of projective transformations which acts cocompactly on $\Omega$.
\end{dfn}

Any such $\Omega$ will come equipped with the Hilbert metric $d_\Omega(x,y)$. To define $d_\Omega(x,y)$ for $x,y \in \Omega$, we consider the line $l$ through $x$ and $y$ and the points $l \cap \partial \Omega = \{a,b\}$. We then have the quadruple of points in the order $a,x,y,b$ and compute the quantity
\begin{equation}
d_\Omega(x,y) = \frac{1}{2} \log \bigg( \frac{|ay| \cdot |xb|}{|ax| \cdot |by|} \bigg)
\end{equation} 

The Hilbert metric comes from a Finsler metric on $\Omega$, see \cite{Marquis}.

\begin{dfn}\label{def:finsler} The natural \textit{Hilbert-Finsler norm} on $\Omega$ is defined by
    \begin{equation}
            F_\Omega(x,v) = \frac{d}{dt} \biggr\rvert_{t = 0} d_\Omega (x, x+tv) = \frac{|v|}{2} \bigg( \frac{1}{|xp^-|} + \frac{1}{|xp^+|} \bigg)
\end{equation}

where $x \in \Omega$, $v \in T_x \Omega$, and $p^+, p^- \in \partial\Omega$ are the forward and backward boundary points in the directions of $v$ and $-v$.
\end{dfn}

For the proof of Theorem \ref{thm:mainminkowski}, we consider the Minkowski functional relative to $\Omega$.

\begin{dfn}\label{dfn:minkowski} The \textit{Minkowski functional} on $\Omega$ is defined by
\begin{equation}\label{equation:minfun}
    M_\Omega(x_0,v) \coloneq \inf \biggl\{a > 0: \frac{1}{a}v \in \Omega \biggr\}
\end{equation}
\end{dfn}
where $x_0 \in \interior \Omega$ and $v \in T_{x_0}\Omega$. The Minkowski functional $M_\Omega(x_0,v)$ can be thought of as the ``forward direction" of the Hilbert-Finsler norm on $\Omega$. Moreover, $F_\Omega(x,v)$ is the symmetrization of $M_\Omega(x_0, v)$.

A foundational result in the study of divisible domains is the following theorem of Benoist.

\begin{thm}[\cite{Benoist}]
   Let $\Omega$ be a divisible convex set divided by a group $\Gamma$. Then the following are equivalent:
   \begin{enumerate}
       \item The metric space $(\Omega,d_\Omega)$ is Gromov-hyperbolic.
       \item The convex set $\Omega$ is strictly convex.
       \item The boundary of $\Omega$ is $C^1$.
       \item The group $\Gamma$ is Gromov-hyperbolic.
   \end{enumerate}
\end{thm}

\subsection{$\alpha$-H{\"o}lder regularity and $\beta$-convexity} \hfill\\

The $\alpha$-H{\"o}lder regularity of $\partial \Omega$ is related to the Anosov geodesic flow of the Hilbert metric on $\Gamma \backslash \Omega$. In this section we recall the statement of this fact due to Benoist and define $\beta$-convexity.

\begin{thm}[\cite{Benoist}]\label{thm:anosov}
    Let $\Gamma$ be a torsion-free discrete subgroup which divides some strictly convex open set $\Omega$ in $\mathbb{S}^n$. Then the geodesic flow $\phi_t$ of the Hilbert metric on the quotient $\Gamma \backslash \Omega$ is Anosov.
\end{thm}

As a consequence of Theorem \ref{thm:anosov}, we have the following statement about the regularity of $\partial \Omega$.

\begin{thm}[\cite{Benoist}]\label{def:alpha}
    Let $\Omega$ be a divisible strictly convex open set in $\mathbb{S}^n$. Suppose that $\Omega$ is not an ellipsoid. Then there exists a maximal $\alpha \in (1,2)$ such that the boundary $\partial \Omega$ is $C^\alpha$.
\end{thm}

We now define $\beta$-convexity.

\begin{dfn}[\cite{Guichard}]\label{dfn:beta} Let $M$ be a hypersurface of class $C^1$ in $\mathbb{R}^n$. $M$ is called \emph{$\beta$-convex} for some real number $\beta \geq 2$  if, for all compact $K$ in $M$, there exists a strictly positive constant $C_K$ such that 
\begin{equation}\label{dfn:betaconvexity}
    d(x,T_yM) \geq C_Kd(x,y)^\beta \text{ for all } x, y \in K.
\end{equation}
\end{dfn}

\begin{rems} \hfill\\
\begin{enumerate}
    \item Guichard remarks in Section 3.2 of \cite{Guichard} that Definition \ref{dfn:beta} applies to the boundary $\partial \Omega$, and is equivalent to $\beta$-convexity of $f$ if $M$ is the graph of a $C^1$ function $f$.
    \item The inequality \ref{dfn:betaconvexity} implies that 
    \begin{equation}\label{dfn:betafunction}
    d(x,T_yM) \geq C_Kt^\beta,
    \end{equation}
    where $t=\|\text{proj}_{T_yM}x\|$.
    \item The constants $\alpha$ and $\beta$ satisfy the relation
    \begin{equation}
        \frac{1}{\alpha} + \frac{1}{\beta} = 1
    \end{equation}
    with $\alpha = \beta = 2$ only when $\Omega$ is an ellipse \cite{Guichard}. In the terminology of \cite{BCL}, $\alpha$ and $\beta$ are \textit{dual indices}.

\end{enumerate}

\end{rems}

For $g \in PGL(V)$, denote by $\lambda_1,\dots,\lambda_{n+1}$ the eigenvalues of $g$ ordered by decreasing modulus and by $l_1,\dots,l_{n+1}$ the logarithms of their norms. If $g$ is biproximal, that is if $g$ and $g^{-1}$ are both proximal, set
\begin{equation}
    \alpha_g \coloneq \frac{l_1-l_{n+1}}{l_1-l_n}, \hspace{1cm} \beta_g \coloneq \frac{l_1-l_{n+1}}{l_1-l_2}
\end{equation}

We define

\begin{equation}
    \beta_\Gamma \coloneq \sup_{g \in \Gamma, g \text{ proximal}}\beta_g \text{    and    } \alpha_\Gamma \coloneq \inf_{g\in\Gamma, g^{-1} \text{ proximal}}\alpha_g
\end{equation}

The numbers $\alpha_\Gamma$ and $\beta_\Gamma$ are related to the regularity of the boundary. Benoist proved that for $\Omega$ a strictly convex divisible domain divided by $\Gamma$, we have
    \begin{equation}
        1 < \alpha \leq \alpha_\Gamma \leq 2 \leq \beta_\Gamma \leq \beta < \infty
    \end{equation}
    where $\alpha$ and $\beta$ are as in Theorem \ref{def:alpha} and Definition  \ref{dfn:beta} above. Guichard showed that the inequalities are in fact equalities:

\begin{thm}[\cite{Guichard}]
    Let $\Omega$ be a strictly convex divisible domain divided by the group $\Gamma$. Denote by $\alpha$ and $\beta$ the regularities of $\Omega$, by $\alpha_\Gamma$ and $\beta_\Gamma$ the constants associated with the group $\Gamma$ as above, and by $a_\phi$ and $a_\phi'$ the constants related to the geodesic flow (see Section 4.4 of \cite{Guichard}). Then
    \begin{enumerate}
        \item The boundary $\partial \Omega$ is $\alpha_\Gamma$-H{\"o}lder; in particular $\alpha = \alpha_\Gamma$
        \item The boundary $\partial \Omega$ is $\beta_\Gamma$-convex; and thus $\beta = \beta_\Gamma$
        \item The contraction constants for the flow are equal: $a_\phi=a_\phi' = \frac{1}{\beta}$
        \item The regularities of $\Omega$ and its dual are equal: $\alpha=\alpha_*, \beta = \beta_*$.
    \end{enumerate}
\end{thm}

\subsection{p-Uniform Convexity}\hfill\\

We now define $p$-uniform convexity in the context of normed spaces. Later, this will apply to the asymmetric Minkowski functional $M_\Omega(x_0,\cdot)$.

\begin{dfn}[Suzuki \cite{Suzuki}]\label{def:puc}
    For $p \in [2, \infty)$, a Banach space $(X, \|\cdot\|)$ is \textit{$p$-uniformly convex} if there exists $C > 0$ satisfying 
    \begin{equation*}
        \delta(\varepsilon) \geq C\varepsilon^p
    \end{equation*}
    for all $\varepsilon \in [0,2]$ where $\delta(\varepsilon)$ is the \textit{modulus of convexity} defined by the infimum
    \begin{equation*}
        \delta(\varepsilon) = \inf \bigg(1-\frac{\|x+y\|}{2}\bigg)
    \end{equation*}
    for $\varepsilon \in [0,2],$ where the infimum is taken over all $x,y \in X$ with $\|x\|, \|y\| \leq 1$, and $\|x-y\| \geq \varepsilon$. 
\end{dfn}

Geometrically, $\delta(\varepsilon)$ can be thought of as measuring the rate at which the midpoint of a chord of length $\varepsilon$ in the unit ball of $X$ approaches the boundary of the ball as $\varepsilon$ decreases.

\begin{rems}\leavevmode
\begin{enumerate}
    \item The notion of $p$-uniform convexity comes from the study of Banach spaces (see \cite{BCL}), and our definition above is due to Suzuki \cite{Suzuki}.
    \item We note that $\delta(\varepsilon)$ is increasing with $\varepsilon$, and that it is sufficient to consider $x, y \in X$ with $\|x\|, \|y\| = 1$.
    \item If $X$ is a Hilbert space, then $\delta(\varepsilon) = 1-\sqrt{1-\frac{\varepsilon^2}{4}}$, so $X$ is 2-uniformly convex.
    \item A space $X$ is called uniformly convex if $\delta(\varepsilon) > 0.$
\end{enumerate}
\end{rems}

For the remainder of this section, we reproduce here the equivalence of formulations of $p$-uniform convexity following \cite{BCL}. These technical details will be useful at the end of Section \ref{sec:proof}:

Let $X$ be a $p$-uniformly convex normed space as in Definition \ref{def:puc}, so that $\delta_X(\varepsilon) \geq \big(\frac{\varepsilon}{C}\big)^p$ for some $C > 0$ and $p>2$. Then with $\|x\|=\|y\|=1$ and $\|x-y\|=\varepsilon$, we have
\begin{equation}
    \norm[\bigg]{ \frac{x+y}{2} }^p \leq \bigg( 1-\bigg(\frac{\varepsilon}{C}\bigg)^p\bigg)^p \leq 1 - \bigg(\frac{p-1}{p}\bigg)^{p-1}\bigg(\frac{\varepsilon}{C}\bigg)^p,
\end{equation}

and thus, with $K=C\big(\frac{1}{q}\big)^{\frac{-1}{q}}, \frac{1}{q}+\frac{1}{p}=1,$
\begin{equation}\label{switcheroo}
   \norm[\bigg] { \frac{x+y}{2} }^p+ \norm[\bigg]{\frac{x-y}{2K}}^p \leq \frac{\|x\|^p+\|y\|^p}{2}
\end{equation}
for all $x$ and $y$ such that $\|x\|=\|y\|$. By replacing $x$ with $x+y$ and $y$ with $x-y$, we find that
\begin{equation}
    \frac{\|x+y\|^p + \|x-y\|^p}{2} \geq \|x\|^p + \|K^{-1}y\|^p
\end{equation}
for all $x$ and $y$ such that $\|x+y\|=\|x-y\|$.

\begin{prop}{\cite{BCL}}\label{switcheroo_prop}
    Let $X$ be a normed space. Then \ref{switcheroo} holds for some constant $K$ and all $x,y \in X$ if and only if $\delta_X(\varepsilon) \geq \big(\frac{\varepsilon}{C}\big)^p$ for some constant $C$.
\end{prop}
\section{Coordinate Set-Up}\label{sec:setup}

Two cases in the proof of the main theorem require linear transformations in order to put $\Omega$ in convenient coordinates. In this section we describe these transformations and prove that the amount they distort $\Omega$ is bounded.

Let $\Omega$ be strictly convex, $x_0 \in \Omega$, and $\xi \in \partial \Omega$. We define

\begin{equation*}
    T: \R^n \longrightarrow \R^n
\end{equation*}

as the composition

\begin{equation*}
    T= T_3 \circ T_2 \circ T_1
\end{equation*}

where

\begin{equation*}
T_1: \R^n  \longrightarrow \R^n
\end{equation*}

rotates $\Omega$ so that $T_\xi \Omega$ is parallel to $\{x_2=0\}$,

\begin{equation*}
T_2: \R^n  \longrightarrow \R^n
\end{equation*}

simultaneously translates $x_0$ to the first-coordinate axis and rotates $\xi$ into the plane given by the first two coordinates, and

\begin{equation*}
T_3: \R^n  \longrightarrow \R^n 
\end{equation*}

is the shear transformation taking $\xi$ to the second coordinate axis, defined by

\begin{equation}
A = \begin{bmatrix} \label{shear}
1      & -\cot(\psi(\xi))      & 0\dots0 \\
0      & 1      &         \\
\vdots & \vdots & I_{n-2} \\
0      & 0      & 
\end{bmatrix}
\end{equation}

where the function 
\begin{equation}
    \psi: \partial \Omega \longrightarrow (0, \pi)
    \end{equation}
is defined as the composition
\begin{equation*}
    \psi = \psi_2 \circ \psi_1
    \end{equation*}
    where
    \begin{align*}
        \psi_1: \partial \Omega & \longrightarrow \mathbb{S}^n \times \mathbb{S}^n\\
        \xi &\longmapsto (v(\xi),w(\xi))
    \end{align*}
    and
    \begin{align*}
        \psi_2: \mathbb{S}^n \times \mathbb{S}^n & \longrightarrow (0,\pi)\\
        (v(\xi),w(\xi)) &\longmapsto \angle_\xi (v(\xi),w(\xi))
    \end{align*}

In the above definition, $v(\xi)$ is the unit vector based at $\xi \in \partial\Omega$ in the direction of the basepoint $x_0 \in \Omega$ and $w(\xi)$ is tangent to $\Omega$ at $\xi \in \partial\Omega$.

\begin{claim}
    For fixed basepoint $x_0 \in \Omega$, the shear transformation defined by $A$ is bounded independently of $\xi$.
\end{claim}

\begin{proof}
    The angle function $\psi$ is continuous on a closed set and thus has a maximum and a minimum. Indeed, $\psi_1$ is continuous since $\partial \Omega$ is $C^1$, and thus $\psi = \psi_2 \circ \psi_1$ is also continuous.
    
    Moreover, $\psi(\xi)$ is bounded away from 0 and $\pi$. We recall that since $\Omega$ is strictly convex, it must lie completely on one side of its tangent line at $\xi \in \partial \Omega$. Then, since $x_0 \in \interior \Omega$ and the vectors $v(\xi)$ and $w(\xi)$ meet at $\xi \in \partial \Omega$, the lines spanned by $v(\xi)$ and $w(\xi)$ cannot be parallel.
\end{proof}

\section{$\beta$-Uniform Convexity of the Hilbert-Finsler Norm}\label{sec:proof}

\begin{thm}
\label{thm:mainminkowski}
Let $\Omega \subset \mathbb{RP}^n$ be a strictly convex divisible domain. Then for any $x_0 \in \Omega,$ the space $\big(T_{x_0}\Omega, M_\Omega(x_0, \cdot)\big)$ is $\beta$-uniformly convex for some $\beta \in [2, \infty)$.
\end{thm}

The approach in the proof of Theorem \ref{thm:mainminkowski} is to work with $\partial \Omega$ as a graph of a $C^1$ function in order to compute an analogous ``modulus of convexity" for a different graph that bounds $\delta_\Omega(\varepsilon)$ from below. The previous section describes how to position $\Omega$ in coordinates where the end result is the creation of orthogonal axes with $x_0$ at the origin. 

The transformation $T: \R^n \longrightarrow \R^n$ grants control of the coordinate axes at the expense of how ``steep" a chord of $\Omega$ sits in the final configuration. We split the proof into cases based on the angle between the chord and the tangent space $T_\xi \Omega$. The degenerate case, where the midpoint $\frac{M_\Omega(x+y)}{2}$ of the chord is aligned with the point $x_0$, is handled separately.

\begin{proof}[Proof of Theorem \ref{thm:mainminkowski}] 
Let $\varepsilon >0$, and consider $x,y \in T_{x_0}\Omega$ with $M_\Omega(x-y) > \varepsilon$. 

\begin{case}[$\frac{x+y}{2}=x_0$]
If $\frac{x+y}{2}=x_0$, then 
\begin{equation*}
    1 - \frac{M_\Omega(x_0,x+y)}{2} = 1
\end{equation*}

Then choose $C_0 = 2^{-\beta}$ so that
\begin{equation}
    C_0\varepsilon^\beta = \bigg(\frac{\varepsilon}{2}\bigg)^\beta \leq 1
\end{equation}
as desired.
\end{case}

When $\frac{x+y}{2} \neq x_0$, we define $\xi \in \partial \Omega$ as the boundary point in the direction of $\frac{x+y}{2}$ as seen from $x_0$. Using $\xi \in \partial \Omega$, we situate $\Omega$ in coordinates as specified in Section \ref{sec:setup}.  We define $\theta$ as the angle between $x-y$ and $T_\xi\Omega$, and consider the cases where $\theta < \frac{\pi}{4}$ and $\theta \geq \frac{\pi}{4}$.

\begin{case}[$\theta < \frac{\pi}{4}$]

    Since $\partial \Omega$ is $\beta$-convex, we have that $\partial \Omega$ is bounded below by $f(t) = C_\Omega|t|^\beta - h$ for some constants $C_\Omega$ and $h$, with both curves meeting at $\xi = (0, -h)$. We denote by $\delta_\Omega (\varepsilon)$ the modulus of convexity of $\Omega$, and define $z$ and $w$ as the vertical projections onto $f(t)$ of $x$ and $y$, respectively, as shown in Figure \ref{fig:smalltheta}. Since $\frac{M_\Omega(x_0,x+y)}{2} < \frac{M_f(x_0,z+w)}{2}$, we proceed by computing $\frac{M_f(z+w)}{2}$ in order to bound $\delta_\Omega(\varepsilon)$ from below.

    \begin{figure}[h]
    \centering
    \includegraphics[width=0.9\textwidth]{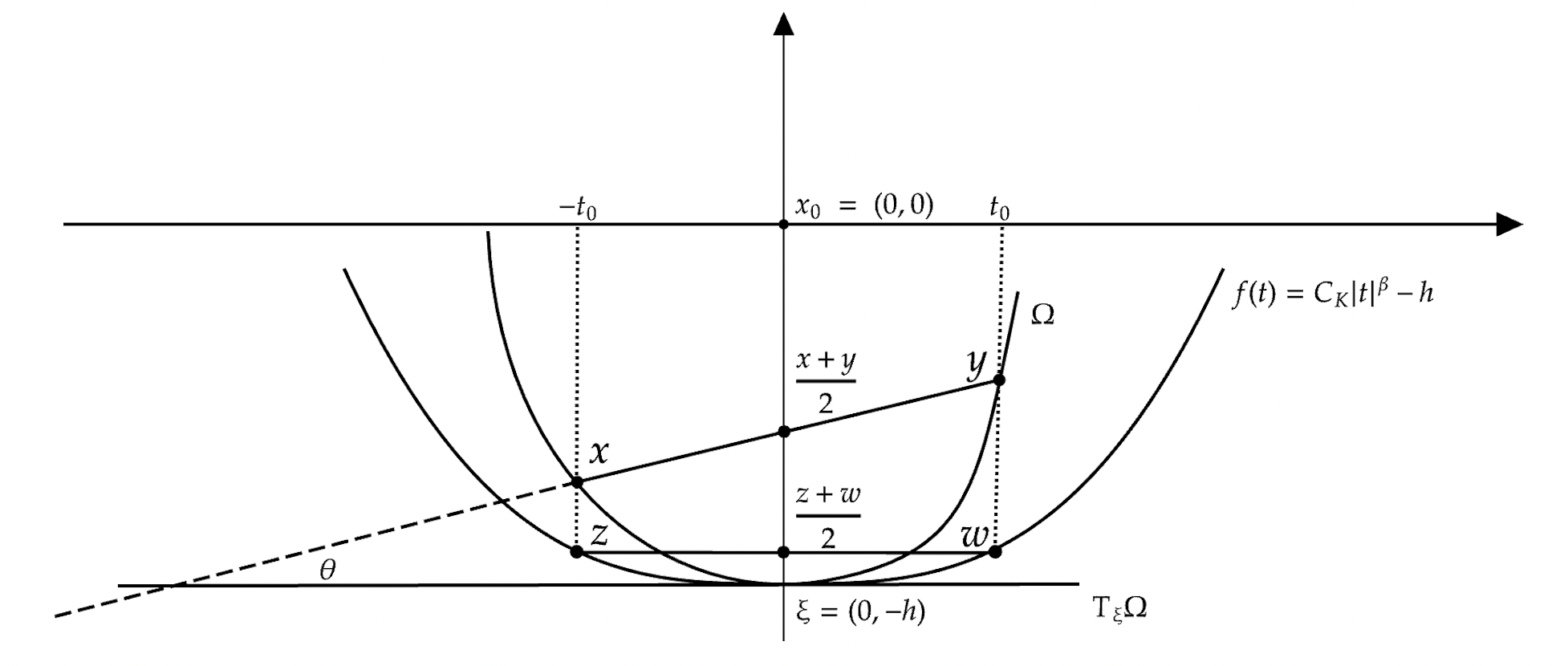}
    \caption{Case 2}
    \label{fig:smalltheta}
\end{figure}

    By definition (see \ref{dfn:minkowski})
    \begin{equation*}
    \frac{M_f(x_0,z+w)}{2} = \frac{1}{a}    
    \end{equation*}
    
     where $a > 0$ is the scaling factor such that $a(C_\Omega|t_0|^\beta - h_1) = -h_1$. Then
    \begin{equation*}
        1 - \frac{M_f(x_0,z+w)}{2} = 1 - \frac{1}{a} = 1 - \frac{C_\Omega|t_0|^\beta -h_1}{-h_1} = \frac{C_\Omega|t_0|^\beta}{h_1}
    \end{equation*}

    Next, we use equivalence of norms to put $\frac{C_\Omega|t_0|^\beta}{h_1}$ in terms of $\varepsilon$, where $\| \cdot \|$ denotes the Euclidean norm. There is a $c >0$ such that

    \begin{align*}
        cM_\Omega(x-y) & \leq \|x-y\| \\
        c\varepsilon & \leq \|x-y\| \\
        c\varepsilon\cos\theta & \leq \|x-y\|\cos\theta \\
        c\varepsilon\cos\theta & \leq 2t_0
    \end{align*}
And thus 
\begin{equation}
    \delta_\Omega (\varepsilon) > \frac{C_\Omega|t_0|^\beta}{h_1} \geq \frac{C_\Omega(c\varepsilon\frac{\sqrt{2}}{2})^\beta}{h_1} = C_1\varepsilon^\beta.
\end{equation}
    
\end{case}

\begin{case}[$\theta \geq \frac{\pi}{4}$]
    We put the piecewise function 
    \begin{equation}
        f(t) = 
\begin{cases}
-h_2 & \text{if } t \leq 0\\
mt-h_2  & \text{if } t > 0,
\end{cases}
 \end{equation} within $\Omega$ as shown in figure \ref{fig:largetheta}.

    \begin{figure}[h]
    \centering
    \includegraphics[width=0.9\textwidth]{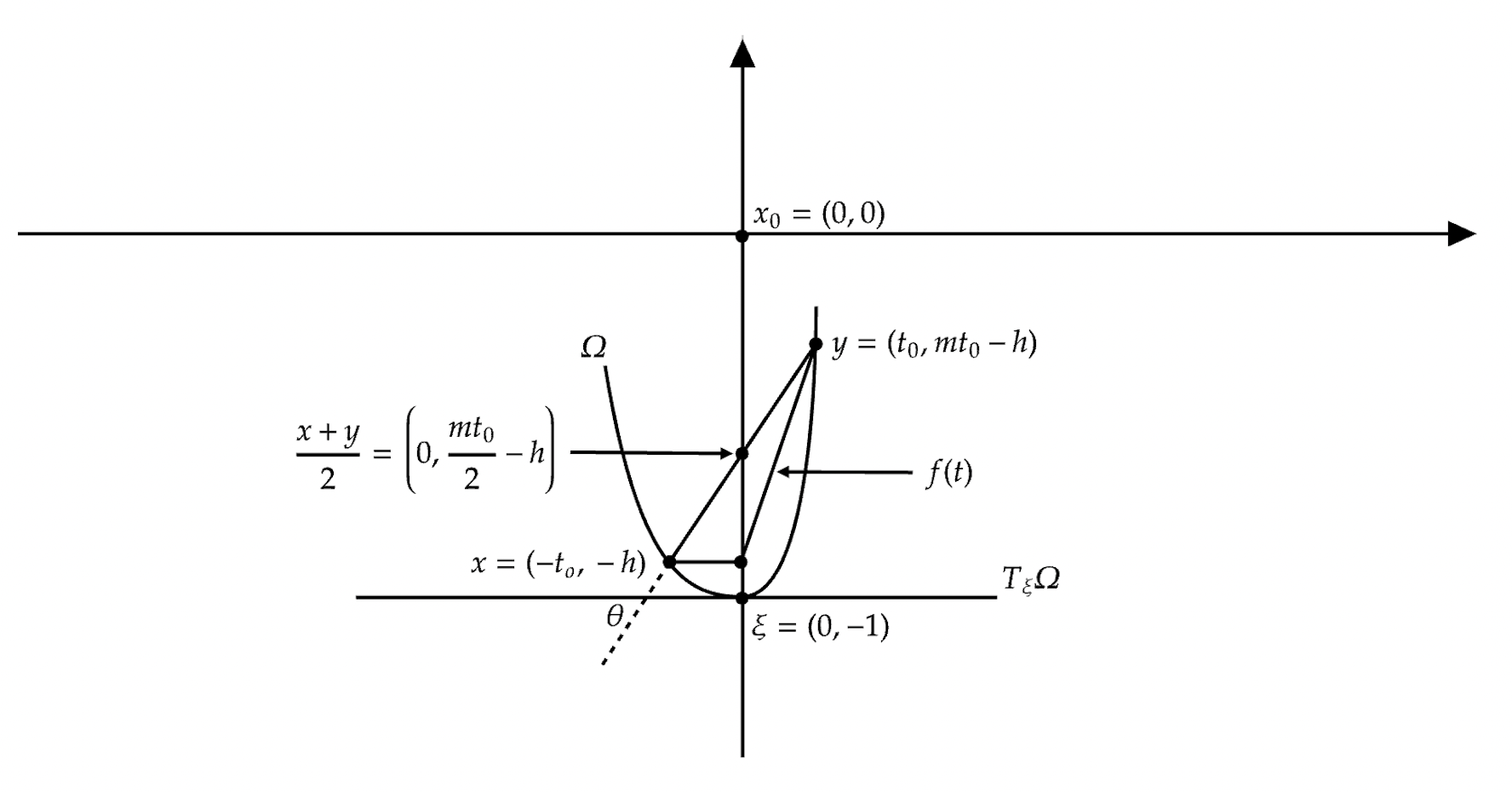}
    \caption{Case 3}
    \label{fig:largetheta}
\end{figure}

We compute $\frac{M_f(x+y)}{2}$ to bound $\delta_\Omega (\varepsilon)$ from below. Here we consider $a > 0$ the scaling factor such that $a(\frac{mt_0}{2} -h_2) = -h_2$. Then

\begin{equation}
    1 - \frac{M_f(x_0,x+y)}{2} = 1 - \frac{1}{a} = 1 - \frac{\frac{mt_0}{2}-h_2}{-h_2} = \frac{mt_0}{2h_2}
\end{equation}

Comparing norms, there is $c > 0$ such that

    \begin{align*}
        cM_\Omega(x-y) & \leq \|x-y\| \\
        c\varepsilon & \leq t_0\sqrt{4+m^2} \\
        \frac{c\varepsilon}{\sqrt{4+m^2}} & \leq t_0
    \end{align*}

So

\begin{equation*}
    \delta_\Omega (\varepsilon) > \frac{mt_0}{2h_2} \geq \frac{mc\varepsilon}{2h_2\sqrt{4+m^2}} \geq \frac{c\varepsilon}{h_2\sqrt{8}} >  \frac{c2^{1-\beta}\varepsilon^\beta}{h_2\sqrt{8}} = C_2\varepsilon^\beta
\end{equation*}

where the third inequality is due to the slope $m$ having a minimum value of 2. This completes Case 3.
\end{case}

Taking $C \coloneq \min \{C_0, C_1, C_2 \}$, we have

\begin{equation}\label{eqn:betauc}
    \delta_\Omega(\varepsilon) \geq C\varepsilon^\beta
\end{equation}
for all $\varepsilon \in [0,2]$.
\end{proof}


We now use the result in \ref{eqn:betauc} to show that Theorem \ref{thm:mainminkowski} holds for the Hilbert-Finsler norm on $\Omega$.

\begin{cor}\label{cor:mainfinsler}
    Let $\Omega \subset \mathbb{RP}^n$ be a strictly convex divisible domain. Then the space $\big(T_{x_0}\Omega, F_\Omega(x_0, \cdot)\big)$ is $\beta$-uniformly convex for some $\beta \in [2, \infty)$.
\end{cor}

\begin{proof}[Proof of Corollary \ref{cor:mainfinsler}]

 We will show that 
 \begin{equation}\label{FisBUC}
        \delta(\varepsilon) = \inf \bigg(1-\frac{F_\Omega(x_0,x+y)}{2}\bigg) \geq C_0 \varepsilon^\beta
    \end{equation}
    for $\varepsilon \in [0,2],$ where the infimum is taken over all $x,y \in T_{x_0}\Omega$ with $F_\Omega(x_0,x), F_\Omega(x_0,y) = 1$, and $F(x_0,x-y) \geq \varepsilon$.

We use the notation 
\begin{equation*}
     M_\Omega^{\pm}(v) \coloneq M_\Omega(\pm v),  v \in T_x\Omega
 \end{equation*}
 throughout.
 There exists an $A > 0$ such that
\begin{equation}\label{equivalencefwdbkwd}
   \frac{1}{A}M_\Omega^-(x_0,\cdot) \leq M_\Omega^+(x_0,\cdot) \leq AM_\Omega^-(x_0,\cdot).
 \end{equation}
 for all $v \in T_x\Omega$. Recall that the Hilbert-Finsler norm $F_\Omega(x_0, \cdot)$ is the symmetrization of the forward and backward Minkowski functionals, that is
 \begin{equation}\label{FisSymmetrization}
     F_\Omega(x_0,\cdot) = \frac{M^+_\Omega(x_0,\cdot) + M^-_\Omega(x_0,\cdot)}{2}.
 \end{equation}
 Using \ref{equivalencefwdbkwd}, we have
 \begin{equation}
  F_\Omega(x_0,\cdot) \leq \frac{M^+_\Omega(x_0,\cdot)}{2} +    \frac{AM^+_\Omega(x_0,\cdot)}{2} = \frac{A+1}{2}M^+_\Omega(x_0,\cdot),
 \end{equation}
and since $F_\Omega(x_0,x-y) \geq \varepsilon$, this implies that
 \begin{equation}\label{unnormalizedlower}
  M^+_\Omega(x_0,x-y) \geq \frac{2\varepsilon}{A+1}.   
 \end{equation}
 And similarly, since $F_\Omega(x_0,x)=F_\Omega(x_0,y)=1,$ we have
 \begin{equation}\label{lowerboundforForward}
   M^+_\Omega(x_0,x) \geq \frac{2}{A+1}  \text{ and } M^+_\Omega(x_0,y) \geq \frac{2}{A+1}.
 \end{equation}
We wish to use the lower bound on the modulus of convexity that we found in \ref{eqn:betauc} for the Minkowski functional in order to find the desired lower bound in \ref{FisBUC} for the Hilbert-Finsler norm. The lower bound in \ref{eqn:betauc} holds for $\varepsilon$-separated vectors $x,y \in T_{x_0}\Omega$ whose forward Minkowski functional values are both less than or equal to 1. Since we are assuming that $F_\Omega(x_0,x)=F_\Omega(x_0,y)=1$, it follows from \ref{FisSymmetrization} that $M^+_\Omega(x_0,\cdot)$ or $M^-_\Omega(x_0,\cdot)$ may be greater than 1. Thus in order to use the lower bound in \ref{eqn:betauc}, we first normalize $x$ and $y$ with respect to $M^+_\Omega(x_0,x)$ and $M^+_\Omega(x_0,y)$.

So let $a = M^+_\Omega(x_0,x)$ and $b = M^+_\Omega(x_0,y)$, and let $a \geq b$. Then $M^-_\Omega(x_0,x) = 2-a$ and $M^-_\Omega(x_0,y) = 2-b$. Define
\begin{equation}
    \overline{x} := \frac{x}{M^+_\Omega(x_0,x)}=\frac{x}{a} \text{ and } \overline{y} := \frac{y}{M^+_\Omega(x_0,y)}=\frac{y}{b}.
\end{equation}
Then $M^+_\Omega(x_0, \overline{x}) = M^+_\Omega(x_0, \overline{y}) = 1$.

Next, we aim to show that $M^+_\Omega(x_0, \overline{x}-\overline{y})$ is bounded below by some multiple of $\varepsilon$. This will allow us to apply Theorem \ref{thm:mainminkowski} to $M^+_\Omega(x_0,\frac{\overline{x}+\overline{y}}{2})$.

Subtracting the backward Minkowski functionals of the normalized $x$ and $y$, we get
\begin{equation}
 M^-_\Omega(x_0, \overline{y})-M^-_\Omega(x_0, \overline{x})= \frac{2-b}{b} - \frac{2-a}{a} = \frac{2(a-b)}{ab}.  
\end{equation}
By the triangle inequality, we have
\begin{equation}
 M^-_\Omega(x_0, \overline{y})-M^-_\Omega(x_0, \overline{x}) \leq M^-_\Omega(x_0, \overline{y}-\overline{x}) = M^+_\Omega(x_0, \overline{x}-\overline{y})
\end{equation}
Then combining the previous two lines gives
\begin{equation}\label{a-b}
    a-b \leq \frac{ab}{2} M^+_\Omega(\overline{x}-\overline{y}) \leq 2 M^+_\Omega(\overline{x}-\overline{y}),
\end{equation}
where the last estimate follows since $a,b \leq 2.$
On the other hand,
\begin{equation}
    x-y = a\overline{x} - b\overline{y} = a\overline{x}-b\overline{x}+b\overline{x}-b\overline{y} = (a-b)\overline{x}+b(\overline{x}-\overline{y}).
\end{equation}
Applying $M^+_\Omega(x_0,\cdot)$, we get
\begin{equation}
    M^+_\Omega(x_0,x-y) \leq (a-b)M^+_\Omega(x_0,\overline{x}) + bM^+_\Omega(x_0,\overline{x}-\overline{y})
\end{equation}
and since $M^+_\Omega(x_0,\overline{x})=1$ and $b \leq 2,$ using the estimate in \ref{a-b} gives
\begin{align}
 M^+_\Omega(x_0,x-y) & \leq bM^+_\Omega(x_0,\overline{x}-\overline{y}) + (a-b)\\
 & \leq 2M^+_\Omega(x_0,\overline{x}-\overline{y}) + 2M^+_\Omega(x_0,\overline{x}-\overline{y})\\
 & = 4M^+_\Omega(x_0,\overline{x}-\overline{y}).
 \end{align}
So by \ref{unnormalizedlower} and the previous line, we get
\begin{equation}
    M^+_\Omega(x_0,\overline{x}-\overline{y}) \geq \frac{\varepsilon}{2(A+1)}=\varepsilon_0
\end{equation}
and hence, by Theorem \ref{thm:mainminkowski},
\begin{equation}
    1-M^+_\Omega\bigg(x_0,\frac{\overline{x}+\overline{y}}{2}\bigg) \geq C\varepsilon_0^\beta.
\end{equation}

We now find bounds for $M^+_\Omega(x_0,\frac{x+y}{2})$ and $M^-_\Omega(x_0,\frac{x+y}{2})$. Taking their average as in \ref{FisSymmetrization} will give the desired result. We begin by observing that
\begin{equation}
    \frac{x+y}{2}=\frac{(a-b)\overline{x}}{2} + \frac{b(\overline{x}+\overline{y})}{2}.
\end{equation}
Applying $M_\Omega^+(x_0,\cdot)$ and simplifying gives
\begin{equation}
 M^+_\Omega\bigg(x_0,\frac{x+y}{2}\bigg)  \leq \frac{a+b}{2} - Cb\varepsilon_0^\beta .
\end{equation}
For the reverse Minkowski functional, ordinary convexity gives
\begin{align}
    M^-_\Omega(x_0,\frac{x+y}{2}) &\leq \frac{M_\Omega^-(x_0,x) + M_\Omega^-(x_0,y)}{2}\\
    & = \frac{(2-a) + (2-b)}{2}\\
    & = 2 - \frac{(a+b)}{2}.
\end{align}
Finally, we take the average
\begin{align}
 F_\Omega\bigg(x_0,\frac{x+y}{2}\bigg) & = \frac{M^+_\Omega(x_0,\frac{x+y}{2}) + M^-_\Omega(x_0,\frac{x+y}{2})}{2} \\
 & \leq \frac{\big(\frac{a+b}{2} - Cb\varepsilon_0^\beta\big) + \big(2 - \frac{a+b}{2}\big)}{2} \\
 & = 1-\frac{Cb\varepsilon_0^\beta}{2}.
\end{align}
And since \ref{lowerboundforForward} implies that $b \geq \frac{2}{A+1}$, we have
\begin{equation}
    1-F_\Omega\bigg(x_0,\frac{x+y}{2}\bigg) \geq \frac{C}{2^\beta(A+1)^{\beta+1}}\varepsilon^\beta
\end{equation}
as desired.


\end{proof}

\section{\textit{Uniform} $\beta$-Uniform Convexity}\label{sec:uniformuniform}

In this section we prove that Theorem \ref{thm:mainminkowski} can be extended to the \textit{uniform} $\beta$-uniformly convex case. That is, we show that the constants obtained in the proof of the main theorem do not depend on the basepoint $x_0 \in K$ for some compact $K \subset \Omega$.

\begin{dfn}
    Let $\beta \geq 2.$ A family of spaces $(X_\alpha, \|\cdot\|_\alpha)$,  $\alpha \in A$ is \textit{uniformly $\beta$-uniformly convex} if  there exists a constant $C > 0$  such that, for all $\alpha \in A$,
    \begin{equation*}
        \delta_\alpha(\varepsilon) \geq C\varepsilon^\beta
    \end{equation*}
    for all $\varepsilon \in [0,2]$, where $\delta_\alpha(\varepsilon)$ is the modulus of convexity defined by the infimum
    \begin{equation*}
        \delta_\alpha(\varepsilon) = \inf \bigg(1-\frac{\|x+y\|_\alpha}{2}\bigg)
    \end{equation*}
    for $\varepsilon \in [0,2],$ where the infimum is taken over all $x,y \in X_\alpha$ with $\|x\|_\alpha, \|y\|_\alpha \leq 1$, and $\|x-y\|_\alpha \geq \varepsilon$. 
\end{dfn}

\begin{prop}Let $\Omega \subset \mathbb{RP}^n$ be a strictly convex divisible domain. Let $K \subset \Omega$ be compact. Then the family 
\begin{equation}
  \biggl\{ \big(T_{x_0}\Omega, M_\Omega(x_0, \cdot)\big): x_0 \in K \biggl\}
\end{equation} is uniformly $\beta$-uniformly convex for some $\beta \in [2, \infty)$.
\end{prop}

\begin{proof}
    We adapt the contents of Sections \ref{sec:setup} and \ref{sec:proof} to the case where all possible basepoints $x_0$ are contained in a compact subset $K \subset \Omega$. Let $K_r$ be the $r$-neighborhood of $K$ for $r = \frac{d_E(K, \partial \Omega)}{2}$, where $d_E$ is the Euclidean distance.


    For a basepoint $x_0 \in \Omega$, Theorem \ref{thm:mainminkowski} shows that for all $\varepsilon \in [0,2]$

    \begin{equation}
     \delta_\Omega(\varepsilon) \geq C\varepsilon^\beta    
    \end{equation}
    for $C \coloneq \min \{C_0, C_1, C_2 \}$, where

    \begin{equation}
        C_0=\frac{1}{2^\beta},
    \end{equation}

        \begin{equation}
        C_1=\frac{C_\Omega(c\frac{\sqrt{2}}{2})^\beta}{h_1},
    \end{equation}

    and

        \begin{equation}
        C_2=\frac{c2^{1-\beta}}{h_2\sqrt{8}}.
    \end{equation}
    
    \begin{claim}\label{uniformshear}
     The coordinate transformation described in Section \ref{sec:setup} is uniformly bounded for all $x_0 \in K$. As a result, the constant $C_\Omega$ can be bounded below by the constant $C_K$ for all $x_0 \in K$.
    \end{claim}

    \begin{proof}[Proof of Claim]
    We recall that the linear transformation $T= T_3 \circ T_2 \circ T_1$ is defined as a composition of a rotation, a translation, and a shear transformation given by the matrix

    \begin{equation}
A = \begin{bmatrix} 
1      & -\cot(\psi(\xi))      & 0\dots0 \\
0      & 1      &         \\
\vdots & \vdots & I_{n-2} \\
0      & 0      & 
\end{bmatrix}
\end{equation}
\vspace{0.1cm}

    where $\psi(\xi)$ is the angle between the vector based at $\xi \in \partial \Omega$ in the direction of the basepoint $x_0$ and the horizontal tangent plane $T_\xi \Omega$.

    We show that $\psi(\xi)$ is bounded away from 0 and $\pi$ for all $\xi \in \partial \Omega$ and for all $x_0 \in K$. Indeed, $K \subset \Omega$ is contained in the cone defined by 

    \begin{equation}
    \theta_1=\arctan \bigg(\frac{r}{\diamop \Omega}\bigg)   
    \end{equation}
    
 and 
 
    \begin{equation}
     \theta_2= \pi - \arctan \bigg(\frac{r}{\diamop \Omega}\bigg).   
    \end{equation}


    It follows that 
    
\begin{equation}
     \frac{-\diamop \Omega}{r} \leq -\cot(\psi(\xi) \leq \frac{\diamop \Omega}{r}   
\end{equation}
and so the shear transformation distorts $\Omega$ by a bounded amount.

    We recall that $\Omega$ is $\beta$-convex, which means

    \begin{equation}
    d(x,T_yM) \geq C_\Omega(x,y)^\beta
\end{equation}

    for all $x, y \in \partial\Omega$. Our choice of shear will leave the vertical distance $d(x,T_yM)$ unchanged, while distorting $d(x,y)$ to at most $d\big(x + \frac{\diamop \Omega}{r},y\big)$.\\
    Since 
    
    \begin{equation}
     d\bigg(x + \frac{\diamop \Omega}{r},y\bigg) \leq \frac{\diamop \Omega}{r} d(x,y),   
    \end{equation}
     we can define $C_K = C_\Omega\big(\frac{r}{\diamop \Omega}\big)^\beta$, which gives us that

     \begin{equation}
    d(x,T_yM) \geq C_Kd(x,y)^\beta
\end{equation}

    for all $x,y\in \Omega$

    \end{proof}

    \begin{claim}\label{uniformc}
      The equivalence of norms constant $c$ used in Cases 2 and 3 of the proof of \ref{thm:mainminkowski} can be replaced by the constant $r = \frac{d_E(K, \partial \Omega)}{2}$, which is uniform for all $x_0 \in K$.
    \end{claim}

    \begin{proof}[Proof of Claim]
      Let $B_E(x_0,r)$ denote the Euclidean ball of radius $r$ at $x_0 \in K$ and let $B_M(x_0,r)$ denote the Minkowski ball of radius $r$ at $x_0 \in K$. Then for all $x_0 \in K$, we have that

      \begin{equation}
       B_E(x_0,r) \subset B_M(x_0,1) = \Omega
      \end{equation}

      and hence

    \begin{equation}
       B_E(x_0,rt) \subset B_M(x_0,t).
      \end{equation}

     Let $v \in \R^n$ and denote by $\|v\|_E$ and $\|v\|_{x_0}$ the Euclidean and Minkowski norms, respectively, of the vector $v$. Set $t = \frac{\|v\|_E}{r}$. Then
    \begin{equation}
       B_E(x_0,\|v\|_E) \subset B_M(x_0,\frac{\|v\|_E}{r}).
      \end{equation}
      So $\|v\|_{x_0} \leq \frac{\|v\|_E}{r}$, or rather $r\|v\|_{x_0} \leq \|v\|_E$ for all $v \in \R^n$ and for all $x_0 \in K$.
    \end{proof}

\begin{claim}
    The constants $h_1$ and $h_2$ can be bounded above by $\diamop \Omega$ for all $x_0 \in K$.
\end{claim}

\begin{proof}[Proof of Claim]
    In Case 2 of the proof of Theorem \ref{thm:mainminkowski}, the constant $h_1$ refers to the distance between $x_0$ and a point $\xi \in \partial \Omega$. We observe that $d(x_0, \xi) < \diamop \Omega$ for all $x_0 \in K$, $\xi \in \partial \Omega$. In Case 3, the constant $h_2$ refers to the distance between $x_0$ and the piecewise function $f(t)$ at $t=0$, which is strictly less than $ d(x_0, \xi)$. Then similarly $h_2 < \diamop \Omega$ for all $x_0 \in K$ and any $\xi \in \partial \Omega$.
\end{proof}

We now use the previous claims to bound each $C_i$ in the proof of Theorem \ref{thm:mainminkowski} from below with uniform constants. For Case 2,

\begin{equation}
    \delta_\Omega (\varepsilon) > C_1\varepsilon^\beta = \frac{C_\Omega(c\varepsilon\frac{\sqrt{2}}{2})^\beta}{h_1} > \frac{C_k(r \varepsilon\frac{\sqrt{2}}{2})^\beta}{\diamop \Omega} = C_1'\varepsilon^\beta .
\end{equation}

For Case 2,

\begin{equation*}
    \delta_\Omega (\varepsilon) > C_2\varepsilon^\beta =  \frac{c2^{1-\beta}\varepsilon^\beta}{h_2\sqrt{8}} > \frac{r2^{1-\beta}\varepsilon^\beta}{\diamop \Omega \sqrt{8}} = C_2'\varepsilon^\beta.
\end{equation*}

Then taking $C' \coloneq \min \{C_0, C_1', C_2' \}$, we have that for all $x_0 \in K$,

\begin{equation}\label{uniformC}
    \delta_\Omega (\varepsilon) > C'\varepsilon^\beta
\end{equation}

for all $\varepsilon \in [0,2]$.
\end{proof}

\begin{cor}\label{ubucFinsler}
    Let $\Omega \subset \mathbb{RP}^n$ be a strictly convex divisible domain with fundamental domain $K$.
    Then the family 
\begin{equation}
  \biggl\{ \big(T_{x_0}\Omega, F_\Omega(x_0, \cdot)\big): x_0 \in K \biggl\}
\end{equation} is uniformly $\beta$-uniformly convex for some $\beta \in [2, \infty)$.
\end{cor}

\begin{proof}
    Since $\Omega$ is divisible, there is a discrete subgroup $\Gamma \subset \mathrm{PGL}(n+1,\mathbb{R})$ of projective transformations acting cocompactly on $\Omega$. We then have a fundamental domain $K$ for this action, and $\Omega = \bigcup_{\gamma \in \Gamma} \gamma(K)$. Since the action is by isometries, we have 
    \begin{equation}
        F_\Omega(\gamma \cdot x, D_{\gamma_x}v) = F_\Omega(x,v)
    \end{equation}
    for all $x \in \Omega, v \in T_x\Omega,$ and $\gamma \in \Gamma$. So, the same $C'$ as in Equation \ref{uniformC} can be used for all translates of $K$, preserving the $\beta$-uniform convexity inequality for all $x_0 \in \Omega$.
\end{proof}

\begin{rem}
    By Proposition \ref{switcheroo_prop}, it follows from Corollary \ref{ubucFinsler} that there exists $C > 0$ such that for any $u,v \in T_{x_0}\Omega$ with $m=\frac{u+v}{2},$
    \begin{equation}\label{Finslerswitch}
        C\cdot F_\Omega(x_0,u-v)^\beta \leq \frac{1}{2}F_\Omega(x_0,u)^\beta + \frac{1}{2}F_\Omega(x_0,v)^\beta - F_\Omega(x_0,m)^\beta
    \end{equation}
    for any $x_0 \in \Omega$.
\end{rem}

\section{Scale-Dependent Local $\beta$-Uniform Convexity of the Hilbert metric}\label{sec:future}


\begin{dfn}[\cite{Ohta}]\label{nonlinearizedpuc}
    Let $(X,d)$ be a nonlinear metric space. Then $(X,d)$ is \textit{2-uniformly convex} if for any $x \in X$ and any minimal geodesic $\eta: [0,1] \longrightarrow X$, we have
\begin{equation}\label{eq:2uc}
    d\bigg(x,\eta\bigg(\frac{1}{2}\bigg)\bigg)^2 \leq \frac{1}{2}d(x,\eta(0))^2 + \frac{1}{2}d(x,\eta(1))^2 - \frac{1}{4C^2}d(\eta(0),\eta(1))^2
\end{equation}
\end{dfn}

\begin{rem}
    When $C=1$, Equation \ref{eq:2uc} corresponds to the CAT(0) property.
\end{rem}

\begin{question}\label{question}
    Let $\Omega$ be a strictly convex divisible domain with the Hilbert metric $d_\Omega$. Is $(\Omega,d_\Omega)$ $\beta$-uniformly convex? 
\end{question}


In this section we show that $(\Omega, d_\Omega)$ is ``scale-locally" $\beta$-uniformly convex. 

The following lemma allows us to compare the distance induced by the Finsler lengths of curves to the distance induced by the Finsler structure as a norm, and will be useful for proving Lemma \ref{GHconvergence}. We remark that similar results are known for low regularity Finsler structures, see Section 3.2 of \cite{LY} and Section 2 of \cite{MT}.

\begin{lemma}\label{closelemma}
    Let $M$ be a $C^1$ Finsler manifold, with Finsler structure $F: TM \rightarrow [0, \infty)$. For $x, y\in M$, let $d(x,y)$ denote the distance  induced by the Finsler length $L_F(c)$ of a $C^1$ curve $c:[0,1] \rightarrow M$ from $x$ to $y$, that is,
    \begin{equation}
        d(x,y) = \inf_cL_F(c)= \inf_c\int_{0}^{1} F(c(t),\dot{c}(t)) \,dt
    \end{equation}
    Let $K \subset M$ be a compact subset contained in a coordinate chart $\phi: U \rightarrow \mathbb{R}^n$. For $x \in K$ and $y \in U$, define the coordinate displacement vector $u_y \coloneq (d\phi)^{-1}(\phi(y) - \phi(x)) \in T_xM$ and define the distance 
    \begin{equation}
    d_x(u_y,u_z) \coloneq F(x, u_z-u_y).
    \end{equation} 
    Then for every $\varepsilon > 0$, there exists $r_0>0$ such that for all $x \in K$ and all $y,z \in B_d(x,r_0),$
    \begin{equation}\label{ineq}
        (1-\varepsilon)d_x(u_y,u_z) \leq d(y,z) \leq (1+\varepsilon)d_x(u_y,u_z)
    \end{equation}
\end{lemma}

\begin{proof}
    In the chart $\phi: U \rightarrow \mathbb{R}^n$, identify $T_pM \simeq \mathbb{R}^n$ via $d\phi_p$. Since norms on $\mathbb{R}^n$ are equivalent and the Finsler structure $F: TM \rightarrow [0, \infty)$ is $C^1$, compactness of $K$ gives constants $0 < m < M < 
    \infty$ such that for all $p \in K$ and $v \in \mathbb{R}^n$,
    \begin{equation}
        m\|v\|_E \leq F(p,v) \leq M\|v\|_E
    \end{equation}
    where $\|\cdot \|_E$ is some Euclidean norm on $\mathbb{R}^n$. Hence on $K$, $d$ is locally bi-Lipschitz equivalent to the Euclidean distance.
    Fix $\varepsilon > 0$ and consider the compact set $\Sigma \coloneq \{(p,v): p \in K, \|v\|_E = 1\} \subset K \times S^{n-1}$. Since the mapping $(p,v) \longmapsto F(p,v)$ is continuous on $\Sigma,$ it is also uniformly continuous on $\Sigma$. So there exists some radius $\rho > 0$ such that if $x\in K, p\in U$ with $|\phi(p)-\phi(x)|<\rho$ and $v \neq 0$, then
    \begin{equation}\label{star}
        (1-\varepsilon)F(x,v) \leq F(p,v) \leq (1+\varepsilon)F(x,v)
    \end{equation}
    Now choose $r_0>0$ small enough so that for any $x\in K$, the $d$-ball $B_d(x,r_0)$ is mapped by $\phi$ into the Euclidean ball $B(\phi(x),\rho)$. For $y,z\in B_d(x,r_0)$, consider the straight segment $\sigma(t) = (1-t)\phi(y)+t\phi(z)$, and let $\gamma = \phi^{-1}\circ \sigma.$ Along the curve $\gamma,$ the basepoint stays within $\rho$ of $x$, so
    \begin{equation}
        L_F(\gamma) = \int_{0}^{1} F(\gamma(t),\dot{\gamma}(t)) \,dt \leq (1+\varepsilon)\int_{0}^{1} F(x,\dot{\gamma}(t)) \,dt
    \end{equation}
    and since $\dot{\gamma}(t)$ corresponds to the constant vector $(d\phi_x)^{-1}(\phi(z)-\phi(y))=u_z-u_y,$ we have
    \begin{equation}
        \int_{0}^{1} F(x,\dot{\gamma}(t)) \,dt =F(x,u_z-u_y).
    \end{equation}
    Taking the infimum over all such curves gives
    \begin{equation}
        d(y,z)\leq L_F(\gamma) \leq (1+\varepsilon)F(x,u_z-u_y).
    \end{equation}
    Now take $\eta$ to be any piecewise $C^1$ curve from $y$ to $z$ contained in $B_d(x,r_0).$ By inequality \ref{star} we have
    \begin{equation}
        L_F(\eta) = \int_{0}^{1} F(\eta(t),\dot{\eta}(t)) \,dt \geq (1-\varepsilon)\int_{0}^{1} F(x,\dot{\eta}(t)) \,dt.
    \end{equation}
    Since $v \longmapsto F(x,v)$ is a norm, for any absolutely continuous curve in $\mathbb{R}^n,$
    \begin{equation}
        \int_{0}^{1} F(x,\dot{\eta}(t)) \,dt \geq F\bigg(x, \int_{0}^{1} \dot{\eta}(t)\bigg) \,dt = F(x,u_z-u_y)
    \end{equation}
    Hence $d(y,z) \geq (1-\varepsilon)F(x,u_z-u_y)$.
\end{proof}

We recall here the definition of pointed Gromov-Hausdorff convergence as in Burago, Burago, and Ivanov (\cite{burago2001course})

\begin{dfn}
    Let $X$ and $Y$ be metric spaces and $f: X \rightarrow Y$ an arbitrary map. The \textit{distortion} of $f$ is defined by
    \begin{equation}
        \text{dis} f=\sup_{x_1,x_2\in X}|d_Y(f(x_1),f(x_2))-d_X(x_1,x_2)|
    \end{equation}
    where $d_X$ and $d_Y$ are the metrics of $X$ and $Y$.
\end{dfn}

\begin{dfn}\label{GHdef}
    A \textit{pointed metric space} is a pair $(X,p)$ consisting of a metric space $X$ and a point $p\in X$.
    A sequence $\{(X_n,p_n)\}_{n=1}^\infty$ of pointed metric spaces \textit{converges in the Gromov-Hausdorff sense} to a pointed space $(X,p)$ if the following holds. For every $r>0$ and $\varepsilon>0$ there exists $n_0 \in \mathbb{N}$ such that for every $n > n_0$ there is a (not necessarily continuous) map $f: B_r(p_n) \rightarrow X$ such that the following hold:
    \begin{enumerate}
        \item $f(p_n)=p$
        \item dis $f<\varepsilon$
        \item the $\varepsilon$-neighborhood of the set $f(B_r(p_n))$ contains the ball $B_{r-\varepsilon}(p)$
    \end{enumerate}
    We use the notation $(X_n,p_n) \xrightarrow[\text{GH}]{\text{}} (X,p)$ for this type of convergence.
\end{dfn}

\begin{lemma}\label{GHconvergence}
    Let $K \subset U \subset \Omega$, where $K$ is compact and $U$ is contained in a coordinate chart. Let $\lambda_n \longrightarrow \infty$ be a sequence of positive real numbers, and let $p_n \in K$ be a sequence such that $p_n \longrightarrow p \in K$. Define
    \begin{equation}
        d_n \coloneq \lambda_nd_\Omega.
    \end{equation}
    
    Then 
    
    \begin{equation}
     (\Omega,d_n,p_n) \xrightarrow[\text{GH}]{\text{}} (T_p\Omega,d_{F_{{\Omega_p}}},0)   
    \end{equation}
    where $d_{F_{{\Omega_p}}}(u,v) \coloneq F_\Omega(p,u- v).$
\end{lemma}

\begin{proof}
    Let $r,\varepsilon>0$. We verify the conditions in Definition \ref{GHdef} on the ball $B_{d_n}(p_n,r).$ Choose a coordinate chart $\phi:U\longrightarrow\mathbb{R}^n$ and identify tangent spaces $T_q\Omega, q \in \Omega,$ with $\mathbb{R}^n.$ Under this identification, $F_\Omega(q,\cdot)$ is a norm on $\mathbb{R}^n$ depending continuously on $q$. 

    Let $0 < \theta< \frac{1}{2}$. By Lemma \ref{closelemma}, applied uniformly to the compact set $K$, there exists $r_0 > 0$ such that for all $q \in K$ and all $y,z \in B_{d_\Omega}(q,r_0),$
    \begin{equation}\label{fromcloselemma}
        (1-\theta)F_\Omega(q, \phi(z)-\phi(y)) \leq d_\Omega(y,z) \leq (1+\theta)F_\Omega(q,\phi(z)-\phi(y)).
    \end{equation}
    Since $\lambda_n \longrightarrow\infty$, for all $n$ sufficiently large we have
    \begin{equation}
        B_{d_n}(p_n,r)=B_{d_\Omega}\bigg(p_n,\frac{r}{\lambda_n}\bigg) \subset B_{d_\Omega}(p_n,r_0) \subset U.
    \end{equation}
    For such $n$, define 
    \begin{align}
        f_n: B_{d_n}(p_n,r) \longrightarrow T_p\Omega \\
        y \longmapsto \lambda_n(\phi(y)-\phi(p_n))
    \end{align}
    Then $f_n(p_n) = 0.$
    Next, we show that \text{dis}$f_n < \varepsilon$. Let $y,z \in B_{d_n}(p_n,r)$ and set
    \begin{equation}
        w_{n,y,z} \coloneq f_n(z) -f_n(y) = \lambda_n(\phi(z)-\phi(y)).
    \end{equation}
    Applying \ref{fromcloselemma} with $q=p_n$ and multiplying by $\lambda_n$, we get
    \begin{equation}\label{scaledfromcloselemma}
    (1-\theta)F_\Omega(p_n, w_{n,y,z}) \leq d_n(y,z) \leq (1+\theta)F_\Omega(p_n,w_{n,y,z}).    
    \end{equation}
    Since $y,z \in B_{d_n}(p_n,r),$ we have $d_n(y,z)\leq 2r$, and using the left hand side of \ref{fromcloselemma}, we have
    \begin{equation}\label{4r}
        F_\Omega(p_n,w_{n,y,z}) \leq \frac{d_n(y,z)}{1-\theta} \leq \frac{2r}{1-\theta} \leq 4r
    \end{equation}
    since $\theta < \frac{1}{2}$. Thus
    \begin{equation}\label{dis4rtheta}
        |d_n(y,z)-F_\Omega(p_n, w_{n,y,z})| \leq \theta F_\Omega(p_n, w_{n,y,z}) \leq 4r\theta.
    \end{equation}
    We now compare $F_\Omega(p_n,\cdot)$ with $F_\Omega(p,\cdot)$. Since $K$ is compact and $F_\Omega$ is continuous, there is a constant $c_K>0$ such that
    \begin{equation}
        F_\Omega(q,w) \geq c_K|w|_E
    \end{equation}
    for all $q \in K$ and all $w\in \mathbb{R}^n$, where $|w|_E$ is a fixed Euclidean norm in the chosen chart. From \ref{4r}, it follows that all vectors $w_{n,y,z}$ under consideration lie in the Euclidean ball
    \begin{equation}
        |w_{n,y,z}|_E \leq \frac{4r}{c_K}.
    \end{equation}
    The set
    \begin{equation}
        \{w\in \mathbb{R}^n:|w|_E \leq \frac{4r}{c_K}\}
    \end{equation}
    is compact, and $p_n \longrightarrow p$, so for $n$ large enough,
    \begin{equation}\label{lessthantheta}
        |F_\Omega(p_n,w)-F_\Omega(p,w)|<\theta
    \end{equation}
    for every $w$ in this ball. Combining \ref{dis4rtheta} and \ref{lessthantheta}, we have
    \begin{align*}
        |d_n(y,z)-d_{F_{\Omega_p}}(f_n(y),f_n(z))| & = |d_n(y,z)-F_\Omega(p,w_{n,y,z})| \\
         & \leq |d_n(y,z)-F_\Omega(p_n,w_{n,y,z})| \\ & \hspace{0.47cm} + |F_\Omega(p_n,w_{n,y,z})-F_\Omega(p,w_{n,y,z})| \\
         & \leq 4r\theta + \theta
    \end{align*}
    Choosing $\theta>0$ so that $4r\theta +\theta < \varepsilon,$ we get dis$(f_n)< \varepsilon$.
    Finally, let $u \in B_{d_{F_{\Omega_p}}}(0,r-\varepsilon),$ so $F_\Omega(p,u)<r-\varepsilon.$ For all $n$ large enough, define
    \begin{equation}
        y_{n,u} \coloneqq \phi^{-1}(\phi(p_n)+\lambda_n^{-1}u.
    \end{equation}
    By construction, $f_n(y_{n,u})=u$. We show that $y_{n,u} \in B_{d_n}(p_n,r).$ First, for $n$ large enough, $y_{n,u} \in B_{d_\Omega}(p_n,r_0).$ So \ref{fromcloselemma}, with $q=p_n, y=p_n,$ and $z=y_{n,u}$ gives
    \begin{equation}
        d_\Omega(p_n,y_{n,u}) \leq (1+\theta)F_\Omega(p_n,\lambda_n^{-1}u).
    \end{equation}
    And multiplying by $\lambda_n,$ we get
    \begin{equation}\label{one}
        d_n(p_n,y_{n,u}) \leq (1+\theta)F_\Omega(p_n,u).
    \end{equation}
    Since $F_\Omega(p,u) < r-\varepsilon$ and $F_\Omega(p_n, \cdot) \longrightarrow F_\Omega(p,\cdot)$ uniformly on the compact set $\overline{B}_{d_{F_{\Omega_p}}}(0,r),$ we can take $n$ large enough so that
    \begin{equation}\label{two}
        F_\Omega(p_n,u) \leq F_\Omega(p,u) + \frac{\varepsilon}{4} < r-\frac{3\varepsilon}{4}
    \end{equation}
    for every $u \in B_{d_{F_{\Omega_p}}}(0,r-\varepsilon).$ Now choose $\theta>0$ so that 
    \begin{equation}\label{three}
        (1+\theta)(r-\frac{3\varepsilon}{4})<r.
    \end{equation}
    Combining \ref{one}, \ref{two}, and \ref{three}, we get
    \begin{equation}
        d_n(p_n,y_{n,u}) \leq (1+\theta)F_\Omega(p_n,u)<(1+\theta)(r-\frac{3\varepsilon}{4}) < r.
    \end{equation}
    So, $y_{n,u} \in B_{d_n}(p_n,r)$.    
    Since $f_n(y_{n,u})=u$, every point of $B_{d_{F_{\Omega_p}}}(0,r-\varepsilon)$ lies in the image $f_n(B_{d_n}(p_n,r)).$ Hence the $\varepsilon-$neighborhood of $f_n(B_{d_n}(p_n,r))$ contains $B_{d_{F_{\Omega_p}}}(0,r-\varepsilon)$. \\
    Since $r,\varepsilon>0$ were arbitrary, $ (\Omega,d_n,p_n) \xrightarrow[\text{GH}]{\text{}} (T_p\Omega,d_{F_{{\Omega_p}}},0).$
\end{proof}


\begin{thm}\label{localishBUC}

Let $\Omega$ be a strictly convex divisible domain. Let $K \subset U \subset \Omega$, where $K$ is compact and $U$ is contained in a coordinate chart. Fix $R \geq 1.$ Then, for $C_0 \in (0,C),$ there exists $r>0$ such that the following holds.

Let $x, y, z\in \Omega$, let $\gamma:[0,1]\longrightarrow  \Omega$ be the geodesic from $y$ to $z$ parametrized by arclength, and let $m=\gamma\big(\frac{1}{2}\big)$. Set $\rho \coloneqq d_\Omega(y,z) < r,$ and assume that $m \in K$ and $d_\Omega(x,m) \leq R\rho$. Then

\begin{equation}\label{localhilbertbuc}
        \frac{C_0}{4}d_\Omega(y,z)^\beta \leq \frac{1}{2}d_\Omega(x,y)^\beta +
\frac{1}{2}d_\Omega(x,z)^\beta  - d_\Omega(x,m)^\beta
\end{equation}
\end{thm}

\begin{proof}
    Let $p \in K$. By \ref{Finslerswitch} we have for any $u,v \in T_p\Omega$ with $m=\frac{u+v}{2},$
    \begin{equation}\label{Finslerineq}
        C\cdot F_\Omega(p,u-v)^\beta \leq \frac{1}{2}F_\Omega(p,u)^\beta + \frac{1}{2}F_\Omega(p,v)^\beta - F_\Omega(p,m)^\beta
    \end{equation}
    Fix $C_0 \in (0,C).$ Since $K \subset U,$ there exists $\eta>0$ such that $B_{d_\Omega}(K,\eta) \subset U$. So choose $r>0$ small enough that $r < \frac{\eta}{R}$. Then whenever $m \in K, \rho \coloneq d_\Omega(y,z) < r,$ and $d_\Omega(x,m) \leq R\rho,$ we have
    \begin{equation}
        d_\Omega(x,m) \leq R\rho < Rr < \eta
    \end{equation}
    and
    \begin{equation}
        d_\Omega(y,m)=d_\Omega(z,m)=\frac{\rho}{2} < r <\eta.
    \end{equation}
    Hence $x,y,z \in U$.
    We proceed by contradiction. Suppose that there is no $r$ such that \ref{localhilbertbuc} holds. That is, for every $n \in \mathbb{N},$ there exist points $x_n,y_n,z_n \in \Omega$ such that 
    \begin{equation}
     m_n \in K, \rho_n := d_\Omega(y_n,z_n) < \frac{1}{n}, \text{and } d_\Omega(x_n,m_n) \leq R\rho_n,   
    \end{equation} but
    \begin{equation}\label{contradiction ineq}
      \frac{C_0}{4}\rho_n^\beta > \frac{1}{2}d_\Omega(x_n,y_n)^\beta +
\frac{1}{2}d_\Omega(x_n,z_n)^\beta  - d_\Omega(x_n,m_n)^\beta .
    \end{equation}
     Define the rescaled metrics $d_n \coloneq \rho_n^{-1}d_\Omega.$ Then
    \begin{equation}\label{oppsideone}
        d_n(y_n,z_n) = 1
    \end{equation}
    and
    \begin{equation}\label{midpointsequences}
        d_n(y_n,m_n) = d_n(z_n,m_n) = \frac{1}{2}.
    \end{equation}
    Since $d_\Omega(x_n,m_n) \leq R\rho_n,$ we have
    \begin{equation}\label{altitudebounded}
        d_n(x_n,m_n) \leq R.
    \end{equation}
    Since $K$ is compact, after passing to a subsequence, we have $m_n \longrightarrow m_\infty \in K$. By Lemma \ref{GHconvergence}, with $p_n = m_n, p=m_\infty,$ and $\lambda_n = \rho_n^{-1}$, we have
    \begin{equation}
     (\Omega,d_n,m_n) \xrightarrow[\text{GH}]{\text{}} (T_{m_\infty}\Omega,d_{F_{{\Omega_{m_\infty}}}},0).   
    \end{equation}
    Fix $R_0 > R+1.$ Then again by Lemma \ref{GHconvergence}, we can choose maps
    \begin{equation}
        f_n:B_{d_n}(m_n,R_0) \longrightarrow T_{m_\infty}\Omega
    \end{equation}
    such that $f_n(m_n)=0$ and dis$(f_n) \longrightarrow 0$. For $n$ large enough, \ref{midpointsequences} and \ref{altitudebounded} imply that $x_n,y_n,z_n \in B_{d_n}(m_n,R_0).$ Now set
    \begin{equation}
        X_n \coloneq f_n(x_n), Y_n \coloneq f_n(y_n), Z_n \coloneq f_n(z_n).
    \end{equation}
    Since dis$(f_n) \longrightarrow 0$, \ref{midpointsequences} and \ref{altitudebounded} imply that $X_n, Y_n,$ and $Z_n$ remain in a bounded subset of $T_{m_\infty}\Omega$. Then after passing to a subsequence, we have
    \begin{equation}
        X_n \longrightarrow X_\infty, Y_n \longrightarrow Y_\infty, \text{ and }Z_n \longrightarrow Z_\infty.
    \end{equation}
    And again since dis$(f_n) \longrightarrow 0$, \ref{oppsideone} and \ref{midpointsequences} give
    \begin{equation}\label{equalsone}
        d_{F_{{\Omega_{m_\infty}}}}(Y_\infty,Z_\infty) =1    
    \end{equation}
    and
    \begin{equation}\label{equalshalf}
        d_{F_{{\Omega_{m_\infty}}}}(0,Y_\infty) = d_{F_{{\Omega_{m_\infty}}}}(0, Z_\infty) = \frac{1}{2}.
    \end{equation}
    Now divide \ref{contradiction ineq} by $\rho_n^{\beta}.$ Using the rescaled metrics $d_n=\rho_n^{-1}d_\Omega,$ we get
    \begin{equation}\label{rescaledineq}
      \frac{C_0}{4} > \frac{1}{2}d_n(x_n,y_n)^\beta +
\frac{1}{2}d_n(x_n,z_n)^\beta  - d_n(x_n,m_n)^\beta . 
    \end{equation}
    Letting $n \longrightarrow \infty$ gives
        \begin{equation}\label{contra}
      \frac{C_0}{4} \geq \frac{1}{2}d_{F_{\Omega_{m_\infty}}}(X_\infty,Y_\infty)^\beta +
\frac{1}{2}d_{F_{\Omega_{m_\infty}}}(X_\infty,Z_\infty)^\beta  - d_{F_{\Omega_{m_\infty}}}(X_\infty,0)^\beta.
    \end{equation}
    We now apply \ref{Finslerineq} at $p=m_\infty$. Set $u \coloneq Y_\infty - X_\infty, v \coloneq Z_\infty - X_\infty,$ so
    \begin{equation}
        \frac{u+v}{2} = \frac{Y_\infty+Z_\infty}{2}-X_\infty=-X_\infty,
    \end{equation}
where in the last equivalence we use that $\frac{Y_\infty + Z_\infty}{2} = 0$, which follows from \ref{equalsone}, \ref{equalshalf}, and strict convexity of the norm. 
    Hence \ref{Finslerineq} gives

    \begin{align*}
      C\cdot F_\Omega(m_\infty,Y_\infty-Z_\infty)^\beta & \leq \frac{1}{2}F_\Omega(m_\infty,Y_\infty-X_\infty)^\beta + \frac{1}{2}F_\Omega(m_\infty,Z_\infty-X_\infty)^\beta \\ & \hspace{0.5cm} - F_\Omega(m_\infty,-X_\infty)^\beta   
    \end{align*}
    Or equivalently, using $d_{F_{\Omega_{m_\infty}}}(Y_\infty, Z_\infty) =1$,
    \begin{equation}\label{diction}
        C  \leq \frac{1}{2}d_{F_{\Omega_{m_\infty}}}(X_\infty,Y_\infty)^\beta + \frac{1}{2}d_{F_{\Omega_{m_\infty}}}(X_\infty,Z_\infty)^\beta - d_{F_{\Omega_{m_\infty}}}(X_\infty,0)^\beta
    \end{equation}
    Since $C_0 \in (0,C),$ we have $C > \frac{C_0}{4}$, so \ref{diction} contradicts \ref{contra}. So, there exists $r>0$  such that \ref{localhilbertbuc} holds whenever $m\in K, \rho < r,$ and $d_\Omega(x,m) \leq R\rho$.
\end{proof}

\bibliographystyle{alpha}
\bibliography{ref}

\end{document}